\documentclass[reqno]{amsart}
\usepackage{amsmath,amsthm,amssymb,amscd}
\newtheorem{theorem}{Theorem}[section]
\newtheorem{lemma}[theorem]{Lemma}

\newtheorem{corollary}[theorem]{Corollary}
\theoremstyle{definition}

\theoremstyle{remark}

\numberwithin{equation}{section}
\allowdisplaybreaks
\begin{document}
\title[Parity distribution and divisibility of Mex-related partition functions]{Parity distribution and divisibility of Mex-related partition functions}

\author[S. Bhattacharyya]{Subhrajyoti Bhattacharyya}
\address{Department of Mathematics, National Institute of Technology Agartala, Tripura, India, PIN- 799046}
\email{jyotisubhra785@gmail.com}

\author[R. Barman]{Rupam Barman}
\address{Department of Mathematics, Indian Institute of Technology Guwahati, Assam, India, PIN- 781039}
\email{rupam@iitg.ac.in}

\author[A. Singh]{Ajit Singh}
\address{Department of Mathematics, Indian Institute of Technology Guwahati, Assam, India, PIN- 781039}
\email{ajit18@iitg.ac.in}

\author[A. K. Saha]{Apu Kumar Saha}
\address{Department of Mathematics, National Institute of Technology Agartala, Tripura, India, PIN- 799046}
\email{apusaha\_nita@yahoo.co.in}
	
\date{March 6, 2023}

%\thanks{We thank Professor Ken Ono for many valuable comments on the article.}



\dedicatory{}
%%% ----------------------------------------------------------------------

\begin{abstract}
	Andrews and Newman introduced the mex-function $\text{mex}_{A,a}(\lambda)$ for an integer partition $\lambda$ of a positive integer $n$ as the smallest positive integer congruent to $a$ modulo $A$ that is not a part of $\lambda$. They then defined $p_{A,a}(n)$ to be the number of partitions $\lambda$ of $n$ satisfying $\text{mex}_{A,a}(\lambda)\equiv a\pmod{2A}$. They found the generating function for $p_{t,t}(n)$ and $p_{2t,t}(n)$ for any positive integer $t$, and studied their arithmetic properties for some small values of $t$. In this article, we study the partition function $p_{mt,t}(n)$ for all positive integers $m$ and $t$. We show that for sufficiently large $X$, the number of all positive integer $n\leq X$ such that $p_{mt,t}(n)$ is an even number is at least $\mathcal{O}(\sqrt{X/3})$ for all positive integers $m$ and $t$. We also prove that for sufficiently large $X$, the number of all positive integer $n\leq X$ such that $p_{mp,p}(n)$ is an odd number is at least $\mathcal{O}(\log \log X)$ for all $m\not \equiv 0\pmod{3}$ and all primes $p\equiv 1\pmod{3}$. Finally, we establish identities connecting the ordinary partition function to $p_{mt,t}(n)$.
	\end{abstract}
%%% ----------------------------------------------------------------------
\maketitle
%%% ----------------------------------------------------------------------
%\tableofcontents
\section{Introduction} 
The minimal excludant or ``mex" function on a set $S$ of positive integers is defined as the least positive integer not in $S$. Andrews and Newman \cite{Andrews-Newman} recently generalized this function to integer partitions. A partition of a non-negative integer $n$ is a non-increasing sequence of positive integers whose sum is $n$. Given a partition $\lambda$ of $n$, they defined the mex-function $\text{mex}_{A,a}(\lambda)$ to be the smallest positive integer congruent to $a$ modulo $A$ that is not part of $\lambda$. Using $\text{mex}_{A,a}(\lambda)$, they next defined the function $p_{A,a}(n)$ as the number of partitions $\lambda$ of $n$ satisfying 
\begin{align*}
\text{mex}_{A,a}(\lambda)\equiv a \pmod{2A}.
\end{align*}
For example, consider $n = 5$, $A = 3$, and $a = 1$. In Table \ref{Table-1}, we list the
seven partitions $\lambda$ of $5$ and the corresponding values of $\text{mex}_{3,1}(\lambda)$ for each $\lambda$.
\begin{table}[ht]
\begin{center}
\begin{tabular}{|c|c|}
\hline
Partition $\lambda$&$\text{mex}_{3,1}(\lambda)$\\
\hline 
$5$&$1$\\
$4+1$&7\\
$3+2$&1\\
$3+1+1$&4\\
$2+2+1$&4\\
$2+1+1+1$&4\\
$1+1+1+1+1$&4\\
\hline 
	\end{tabular}
\caption{Values of $\text{mex}_{3,1}(\lambda)$}
\label{Table-1}
\end{center}
\end{table}
We see that three of the partitions of $5$ satisfy $\text{mex}_{3, 1}(\lambda)\equiv 1\pmod{6}$. Therefore, $p_{3, 1}(5)=3$.
Andrews and Newman \cite[Lemma 9]{Andrews-Newman} proved that the generating function for $p_{t,t}(n)$ is given by
\begin{align}\label{gen-fun}
\sum_{n=0}^{\infty}p_{t,t}(n)q^n=\frac{1}{(q; q)_{\infty}}\sum_{n=0}^{\infty}(-1)^n q^{tn(n+1)/2}
\end{align}
and the generating function for $p_{2t,t}(n)$ is given by
\begin{align}\label{gen-fun1}
\sum_{n=0}^{\infty}p_{2t,t}(n)q^n=\frac{1}{(q; q)_{\infty}}\sum_{n=0}^{\infty}(-1)^n q^{tn^2},
\end{align}
where $\displaystyle (a; q)_{\infty}:= \prod_{j=0}^{\infty}(1-aq^j)$. 
\par Using the generating functions and elementary $q$-series manipulations, Andrews and Newman \cite{Andrews-Newman} proved that $p_{1, 1}(n)$ equals the number of partitions of $n$ with non-negative crank and $p_{3,3}(n)$ equals the number of partitions of $n$ with rank $\geq -1$. They also proved that $p_{2,1}(n)$ is equal to the number of partitions of $n$ into even parts. They further proved that $p_{4,2}(n) - p_o(n)$ equals the number of partitions of $n$ into parts congruent to
$\pm 4, \pm 6, \pm 8, \pm 10$ modulo $32$ and $p_{6,3}(n) - p_o(n)$ equals the number of partitions of $n$ into parts congruent to $\pm 2, \pm 4, \pm 5, \pm 6, \pm 7, \pm 8$ modulo $24$, where $p_o(n)$ denotes the number of partitions of $n$ into odd parts. 
\par In a recent paper \cite{BS2}, the second and the third author have established identities connecting the ordinary partition function to $p_{t,t}(n)$ and $p_{2t,t}(n)$; and the Andrews' singular overpartition function to $p_{t,t}(n)$ for all $t\geq 1$. In another recent paper \cite{BS1}, the second and the third author have proved that $p_{2^{\alpha},2^{\alpha}}(n)$ and $p_{3\cdot2^{\alpha}, 3\cdot2^{\alpha}}(n)$ are almost always even for all $\alpha\geq 1$ using the theory of modular forms and $\eta$-quotients. 
\par 
In this article, we study the partition function $p_{mt,t}(n)$ for all positive integers $m$ and $t$. We first find the generating function, and then study the distribution of $p_{mt,t}(n)$ modulo $2$. In the following theorem, we obtain quantitative estimate for the distribution of even values of $p_{mt,t}(n)$ for all  positive integers $t$ and $m$.
\begin{theorem}\label{thm2}
	Let $t$ and $m$ be positive integers. Then, for large $X$, we have 
	\begin{align}\label{bound2}
	\#\left\lbrace n\leq X : p_{mt,t}\left( n\right)\equiv 0\pmod{2} \right\rbrace \gg \sqrt{X/3}.
	\end{align}
\end{theorem}
 Next, we obtain quantitative estimate for the distribution of odd values of $p_{mt,t}(n)$ for certain positive integers $t$ and $m$.
 \begin{theorem}\label{thm5}
 	Let $m$ be a positive integer and $p$ be a prime such that $m \not \equiv 0 \pmod{3}$ and $p\equiv 1\pmod{3}$. Then, for large $X$, we have
 	\begin{align*}
 	\# \{ n \leq X : p_{mp,p}(n)\equiv 1\pmod{2} \} \geq \beta \log\log X,
 	\end{align*} 
 	where $\beta > 0$ is a constant.
 \end{theorem}
Finally, we establish identities connecting the ordinary partition function $p(n)$ to $p_{mt,t}(n)$ for all positive integers $t$ and $m$ in Section \ref{final-part}. Using these identities, we prove that the Ramanujan's famous congruences for $p(n)$ are also satisfied by $p_{mt,t}(n)$.
\section{Parity distribution of $p_{mt,t}(n)$} 
In this section, we first prove the generating functions of $p_{mt,t}(n)$  for all positive integers $t$ and $m$.
\begin{lemma}\label{lem1}
 Let $t$ and $m$ be positive integers. Then
 \begin{align}\label{lem1.1}
\sum_{n=0}^{\infty}p_{mt,t}(n)q^{n}= \frac{1}{(q;q)_{\infty}}\sum_{n=0}^{\infty}(-1)^{n} q^{\frac{1}{2}(mn^{2}-(m-2)n)t}.
\end{align}
 \end{lemma}
 \begin{proof}
 	We have 
 \begin{align*}
&\frac{1}{(q;q)_{\infty}}\sum_{n=0}^{\infty}(-1)^{n} q^{\frac{1}{2}(mn^{2}-(m-2)n)t}\\
&= \frac{1}{(q;q)_{\infty}}\sum_{n=0}^{\infty} ( q^{\frac{1}{2}(4n^{2}m -2n(m-2))t} - q^{\frac{1}{2}((4n^{2}+4n+1)m - (2n+1)(m-2))}) \\
&= \frac{1}{(q;q)_{\infty}}\sum_{n=0}^{\infty} q^{\frac{1}{2}(4n^{2}m -2n(m-2))t}(1-q^{2mn+1)t})\\
&=\frac{1}{(q;q)_{\infty}}\sum_{n=0}^{\infty} q^{t+(m+1)t+(2m+1)t+\cdots+((2n-1)m+1)t}(1-q^{(2mn+1)t})\\
&= \sum_{n=0}^{\infty}\frac{ q^{t+(m+1)t+(2m+1)t+\cdots+((2n-1)m+1)t}}{\prod_{r=1, r \not= t(2mn+1)}^{\infty}(1-q^{r})}.
\end{align*}
The last expression is clearly the generating function for $p_{mt,t}(n)$.
 \end{proof} 
We readily obtain \eqref{gen-fun} and \eqref{gen-fun1} by taking $m=1, 2$ in Lemma \ref{lem1}, respectively.
 %%%%%%%%%%%%%%%%%%%%%%%%%%%%%%%%%%%%%%%%%
 \subsection{Parity distribution of $p_{mt,t}(n)$: even case }
  In this section, we prove Theorem \ref{thm2}. In \cite{ahlgren}, Ahlgren found quantitative estimates for the distribution of parity of the ordinary partition function $p(n)$ in arithmetic progression. Our proof of Theorem \ref{thm2} is inspired by Ahlgren \cite{ahlgren}.  
\begin{proof}[Proof of Theorem \ref{thm2}]
We first recall, Euler's Pentagonal Number Theorem \cite[(1.3.18)]{berndt},
\begin{align}\label{Pentagonal}
(q; q)_{\infty}=\sum_{n\in \mathbb{Z}}^{}(-1)^{n}q^{\frac{n}{2}(3n-1)}.
\end{align}
Employing \eqref{Pentagonal} in \eqref{lem1.1}, and then taking modulo $2$ we obtain
\begin{align}\label{thm2.1}
\sum_{n=0}^{\infty}q^{\frac{1}{2}(mn^{2}-(m-2)n)t} \equiv \sum_{n \in \mathbb{Z}} q^{\frac{n(3n-1)}{2}} \sum_{n=0}^{\infty}p_{mt,t}(n)q^{n}\pmod{2}.
\end{align}
We define $u_{k} = \frac{1}{2}k(3k-1)$ for all $k \in \mathbb{Z}$. Then, for every $n$ we define 
\begin{align*}
\mathcal{N}_{n} :=  \{ n- u_{k} : 0 \leq u_{k} \leq n\text{~for some~}k \in \mathbb{Z} \}.
\end{align*}
Clearly for
\begin{align*}
\sum_{n=0}^{\infty}b(n)q^{n}:=\sum_{n=0}^{\infty} q^{\frac{1}{2}(mn^{2}-(m-2)n)t},
\end{align*}
we have
\begin{align}\label{b6ox1}
\#\left\lbrace n\leq X : b(n)\ \text{is odd}\right\rbrace=o(X).
\end{align}
Now, comparing the coefficients of $q^{n}$ on both sides of \eqref{thm2.1}, we obtain
\begin{equation}\label{thm2.2}
b(n)\equiv\sum_{s\in\mathcal{N}_n}p_{mt,t}(s) \pmod{2}.
\end{equation}
Note that for $k\geq1$, if $u_{-(k-1)}\leq n< u_k$, then $|\mathcal{N}_n|=2k-1$ and if $u_{k}\leq n< u_{-k}$, then $|\mathcal{N}_n|=2k$. 
Thus, $|\mathcal{N}_n|$ is odd if and only if $n$ is in an interval of the form $B_{k}:=\left[u_{-(k-1)},u_{k}\right)$. There exists a positive constant $D$ such that $B_{k}\subset\left[0,X \right] $, $0\leq k\leq D\sqrt{X}$, for large $X$. 
The fact that the length of $B_{k}$ is $\gg k$ implies
\begin{equation*}
\#\left\lbrace n\leq X :n\in B_k\  \text{for some}\  k\right\rbrace \gg \sum_{k=0}^{D\sqrt{X}}k\gg X.
\end{equation*}
Therefore, $\#\left\lbrace n\leq X :|\mathcal{N}_n|\ \text{is odd}\  \right\rbrace \gg X$, and together with \eqref{b6ox1} we conclude that 
\begin{equation*}
\#\left\lbrace n\leq X :|\mathcal{N}_n|\ \text{is odd,}\ b(n)\ \text{is even} \right\rbrace \gg X.
\end{equation*}
It is clear from \eqref{thm2.2} that for every $n\in \left\lbrace n\leq X :n\in B_k\  \text{for some}\  k\right\rbrace$, $p_{mt,t}(s)$ is even for some $s\in \mathcal{N}_n$. This gives
\begin{equation*}
\#\left\lbrace(n,s): n\leq X,\ s\in \mathcal{N}_n,\ p_{mt,t}(s)\ \text{is even} \right\rbrace \gg X.
\end{equation*}
We now wish to count $N_{s, X}:=\#\left\lbrace n\leq X: s\in \mathcal{N}_n \right\rbrace$. For fixed $s$, $N_{s, X}$ is not more than $\#\left\lbrace k\in\mathbb{Z}: 0\leq u_k\leq X \right\rbrace$, and this number is clearly $\ll \sqrt{3X}$. 
Therefore, $N_{s, X} \ll \sqrt{3X}$, and we arrive at \eqref{bound2}. This completes the proof of the theorem.
\end{proof}
%%%%%%%%%%%%%%%%%%%%%%%%%%%%%%%%%%%%%%%
\subsection{Parity distribution of $p_{mt,t}(n)$: odd case }
In this section, we prove Theorem \ref{thm5}. In order to prove Theorem \ref{thm5}, we first prove the following lemmas. Our proof is inspired by Kolberg \cite{kolberg} and Ray \cite{ray}.
\begin{lemma}\label{lem3}
Let $t$ and $m$ be positive integers. Then, for any positive integer $n$, we have 
\begin{align*}
&\sum_{s=0}^{\infty}{p_{mt,t}\left(n-\frac{s(3s-1)}{2}\right)} + \sum_{s=1}^{\infty}{p_{mt,t}\left(n-\frac{s(3s+1)}{2}\right)}\\
&\equiv 
\begin{cases}
1 \pmod{2}, & \mbox{if $n=\frac{1}{2}((mk^{2} -(m-2)k)t)$ for some $k \in \mathbb{N}$};\\
0\pmod{2}, & \mbox{otherwise}.\nonumber
\end{cases}
\end{align*}
\end{lemma}
\begin{proof}
By \eqref{Pentagonal}, we have 
\begin{align}\label{lem3.l}
(q; q)_{\infty}\equiv\sum_{n=0}^{\infty}q^{\frac{n}{2}(3n-1)}+\sum_{n=1}^{\infty}q^{\frac{n}{2}(3n+1)}\pmod2.
\end{align}
Employing \eqref{lem3.l} in \eqref{lem1.1}, we obtain
\begin{align*}
\sum_{n=0}^{\infty}{p_{mt,t}(n)q^{n}}  &= \frac{1}{(q;q)_{\infty}}\sum_{n=0}^{\infty} (-1)^{n}q^{\frac{1}{2}(mn^{2}-(m-2)n)t}\\
 & \equiv \frac{1}{(q;q)_{\infty}}\sum_{n=0}^{\infty}q^{\frac{1}{2}(mn^{2}-(m-2)n)t} \pmod{2}\\
& \equiv \frac{\sum_{n=0}^{\infty} q^{\frac{1}{2}(mn^{2}-(m-2)n)t}}{\sum_{n=0}^{\infty} {q^{\frac{n(3n-1)}{2}}+\sum_{n=1}^{\infty}} {q^{\frac{n(3n+1)}{2}}}} \pmod{2}.
\end{align*}
Hence,
\begin{align*}
&\sum_{n=0}^{\infty}\left(\sum_{s=0}^{\infty}{p_{mt,t}\left(n-\frac{s(3s-1)}{2}\right)} + \sum_{s=1}^{\infty}{p_{mt,t}\left(n-\frac{s(3s+1)}{2}\right)}\right)q^n\\
& \equiv \sum _{n=0}^{\infty}q^{\frac{1}{2}(mn^{2}-(m-2)n)t} \pmod{2}.
\end{align*}
Now, for any non-negative integer $n$,  comparing the coefficients of $q^n$ on both sides of the above congruence yields the required expression.
\end{proof}
%%%%%%%%%%%%%%%%%%%%%%%%%%%%%%%%%%%
\begin{lemma}\label{lem5}
Let $t, m, r$ be positive integers with $r\geq 2$. If $r(3r-1)$ is not of the form $(mk^{2}-(m-2)k)t$ for any positive integer $k$, then there exists an integer $n \in [2r-1,\frac{r(3r-1)}{2}]$ such that $p_{mt,t}(n)$ is odd.
\end{lemma}
\begin{proof}
 We prove the lemma by using the method of contradiction. We consider $r \geq 2$ such that $r(3r-1)$ is not of the form $ (mk^{2}-(m-2)k)t$ for any $k\in \mathbb{N}$. If possible let $p_{mt,t}(n)$ is even for any $n\in [2r-1, \frac{r(3r-1)}{2}]$.
\par For an integer $a$, let 
\begin{align*}
\mathcal{S}(a)&:= \frac{r(3r-1)}{2} - \frac{k(3a-1)}{2},\\
\mathcal{R}(a)&:= \frac{r(3r-1)}{2}- \frac{a(3a+1)}{2}.
\end{align*}
By Lemma \ref{lem3}, we have
\begin{align}\label{lem5.1}
\sum_{a=0}^{\infty}{p_{mt,t}(\mathcal{S}(a))} + \sum_{a=1}^{\infty}{p_{mt,t}(\mathcal{R}(a))}\equiv 0\pmod{2}.
\end{align}
It is easy to check that  $\mathcal{S}(a)<0$, if $a\geq r+1$ and for $a\geq r$, $\mathcal{R}(a)<0$. We adopt the convention that $p_{mt,t}(n)=0$ when $n$ is a negative integer. Hence, by \eqref{lem5.1}, we can truncate the series into finite sums of the form 
\begin{align}\label{lem5.2}
&\sum_{a=0}^{r}{p_{mt,t}(\mathcal{S}(a))} + \sum_{a=1}^{r-1}{p_{mt,t}(\mathcal{R}(a))}
= 1+ \sum_{a=0}^{r-1}{p_{mt,t}(\mathcal{S}(a))} + \sum_{a=1}^{r-1}{p_{mt,t}(\mathcal{R}(a))}.
\end{align}
Now, for any fixed positive integer $r\geq 2$, $\mathcal{S}(a)$ is a decreasing function of $a$. We note that $\mathcal{S}(0)=\frac{r(3r-1)}{2}$ and 
$$\mathcal{S}(r-b)= \frac{1}{2}(6rb-3b^2 -b)= \frac{3b}{2}\left(2r-\left(b+\frac{1}{3}\right)\right)\geq 2r-1, $$
where $b\in  \{1,2,\ldots, r-1\}$. Hence, $\mathcal{S}(a) \in [2r-1,\frac{r(3r-1)}{2}]$ for $a\in  \{0,1,\ldots, r-1\}$.
Similarly, we can show that $\mathcal{R}(a) \in [2r-1,\frac{r(3r-1)}{2}]$ for each $a\in  \{1,2,\ldots, r-1\}$. By our assumption $p_{mt,t}(n)$ is even when $n\in [2r-1,\frac{r(3r-1)}{2}] $. Hence, $\sum_{a=0}^{r-1}{p_{mt,t}(\mathcal{S}(a))}$ and $\sum_{a=1}^{r-1}{p_{mt,t}(\mathcal{R}(a))}$ are even numbers and consequently, the summation \eqref{lem5.2} is an odd number, which is a contradiction to \eqref{lem5.1}. This completes the proof of the lemma.
\end{proof} 
%%%%%%%%%%%%%%%%%%%%%%%%%%%%%%%%%%%%
In Lemma \ref{lem5}, we have seen that if $r(3r-1)$ is not of the form $(mk^{2}-(m-2)k)t$ for any positive integer $k$, then there exists an integer $n \in [2r-1,\frac{r(3r-1)}{2}]$ such that $p_{mt,t}(n)$ is odd. In the following lemma, we prove that, if $r \equiv  2\pmod{3}$ then $r(3r-1)\neq (mk^{2}-(m-2)k)t$ for any positive integers $m, k$ with $m \not \equiv 0 \pmod{3}$ and any prime $p\equiv 1\pmod{3}$.
\begin{lemma}\label{lemma-new}
	Let $s$ be a positive integer such that $s \equiv  2\pmod{3}$. Then, $s(3s-1)$ is not of the form $(mk^{2} -(m-2)k)p$ for any positive integers $m, k$ with $m \not \equiv 0 \pmod{3}$ and any prime $p\equiv 1\pmod{3}$.
\end{lemma}
\begin{proof}
We prove the lemma by the method of contradiction. For a given positive integer $s \equiv  2\pmod{3}$, suppose that $s(3s-1) = pmk^{2}-(m-2)kp$ for some positive integers $k, m$ and prime $p$ with $m \not \equiv 0 \pmod{3}$ and $p\equiv 1\pmod{3}$. Then, $(m-2)^{2}p^{2} + 4mps(3s-1)$ must be a square of an integer. Hence, there exists a positive integer $\nu$ such that 
\begin{align}\label{thm5.1}
\nu((m-2)p+\nu)= mps(3s-1).
\end{align}
 Observe that $ p|\nu$ and hence, $\nu= pu $ for some $u\in \mathbb{N}$. Thus, \eqref{thm5.1} yields $p^{2}u(u+m-2)=mps(3s-1)$. Since $p\equiv 1\pmod3$, we have
\begin{align}\label{thm5.2}
 u(u+m-2)\equiv 2ms\pmod{3}.
\end{align}
Now, if $m \equiv 1 \pmod{3}$ then $u(u-1) \equiv 2s\pmod{3}$.
 If $u \equiv 1\pmod{3}$, then $u-1 \equiv 0\pmod{3}$ and hence, $u(u-1)\equiv 0\pmod{3}$, which contradicts the fact \eqref{thm5.2}, as $s\equiv 2\pmod3$. And if $ u\equiv 2\pmod{3}$ then $u-1 \equiv 1\pmod{3}$ then we have $u(u-1) \equiv 2\pmod{3}$ which again gives a contradiction to the fact \eqref{thm5.2}. Again if $m \equiv 2 \pmod{3}$ then $u^{2} \equiv s \pmod{3}$  as $ 3 \not| \nu$, $u^{2} \equiv 1\pmod{3}$, which gives a contradiction to  \eqref{thm5.2}. This completes the proof of the lemma.
\end{proof} 
Combining Lemma \ref{lem5} and Lemma \ref{lemma-new}, we readily obtain that, if $m$ is a positive integer and $p$ is a prime such that $m \not \equiv 0 \pmod{3}$ and $p\equiv 1\pmod{3}$, then $p_{mp, p}(n)$ is odd for infinitely many integer $n$. We now prove Theorem \ref{thm5} which gives a quantitative estimate for the distribution of odd values of $p_{mt,t}(n)$ when $m \not \equiv 0 \pmod{3}$ and $p\equiv 1\pmod{3}$.
\begin{proof}[Proof of Theorem \ref{thm5}]
To prove our theorem, we use Lemmas \ref{lem5} and \ref{lemma-new}. Let $n$ be a positive integer. We want to count the number of elements in the set 
$$ \{ 1\leq n \leq X : p_{mp,p}(n)~~ \text{ is an odd integer} \}.$$
  We next define $a_{k}$, for $k\geq 0$, recursively as follows. 
\begin{align}\label{thm5.3}
 a_{0}:= s~~ \text{and }~~a_{k}:= \frac{1}{2}a_{k-1}(3a_{k-1}-1),~~ \text{for} ~~k\in\mathbb{N}.
 \end{align}
We note that for all non-negative integers $k$, $a_{k}\equiv 2\pmod{3}$ and $a_{k}$ is a strictly increasing sequence, where $a_{k+1}- (2a_{k} -1) \geq 2$. We now partition the interval $[1,X]$ as follows.
$$ [1,X]= [1,a_{1})\cup[a_{1},a_{2})\cup \cdots\cup [a_{\nu},X],$$
where $\nu$ is the largest integer such that $a_{\nu} \leq X$. By Lemma \ref{lemma-new}, we obtain that for a fixed positive integer $s$ with $s \equiv  2\pmod{3}$, $s(3s-1) \not= pmk^{2}-(m-2)kp$ for any positive integers $m,k$ such that $m\not\equiv 0\pmod3$ and any prime $p\equiv 1\pmod3$. Hence, by Lemma \ref{lem5}, we can find a positive integer $n\in [2a_{k}-1, a_{k+1}] \subset [a_{k},a_{k+1}]$ for which $p_{mp,p}(n)$ is an odd integer. Then the number of $n\leq X$ for which $p_{mp,p}(n)$ is an odd integer is atleast $\lfloor \frac{\nu}{2}\rfloor$.  Now, for all $k\geq 0$, we have
$$ a_{k} = \frac{a_{k-1}(3a_{k-1}-1)}{2}\leq \frac{3}{2}a_{k-1}^{2}\leq \cdots\leq \left(\frac{3}{2}\right)^{2^{k-1}-1}\leq 2^{2^{k}}.$$
Since $a_{\nu}\leq X\leq a_{\nu+1}$, we find that $\nu \geq \beta \log\log X$ for some constant $\beta>0$. This completes the proof of the theorem.
\end{proof} 
%%%%%%%%%%%%%%%%%%%%%%%%%%%%%%%%%%%%%%%%%%%%%%%%%%%%%%%
\section{Mex-related partitions and relations to ordinary partition}\label{final-part}
Let $p(n)$ denote the ordinary partition function. We adopt the convention that $p(n)=0$ when $n$ is a negative integer. In the following theorem, we express $p_{mt,t}(n)$ in terms of $p(n)$.
 \begin{theorem}\label{thm3}
 	Let $t, m$ be positive integers. Then, for all non-negative integers $n$, we have
 	\begin{align}\label{thm3.1}
 	p_{mt,t}(n)= p(n)+\sum_{r=1}^{\infty}p(n-tr(2rm-m+2))-\sum_{s=1}^{\infty}p(n-t(2s-1)(sm-m+1)).
 	\end{align}
 \end{theorem}
 \begin{proof}
 	The generating function for the partition function $p(n)$ is given by
 	\begin{align*}
 	\sum_{n=0}^{\infty} p(n) q^{n}=\frac{1}{(q ; q)_{\infty}}.
 	\end{align*}
 	From \eqref{lem1.1}, we obtain that
 	\begin{align*}
 	&\sum_{n=0}^{\infty}p_{mt,t}(n)q^n\\
 	&=\frac{1}{(q; q)_{\infty}}\sum_{n=0}^{\infty}(-1)^nq^{\frac{1}{2}(mn^{2}-(m-2)n)t}\\
 	&=\left(\sum_{n=0}^{\infty} p(n) q^{n}\right)\left(1+\sum_{r=1}^{\infty} q^{tr(2rm-m+2)}-\sum_{s=1}^{\infty} q^{t(2s-1)(sm-m+1)}\right)\\
 	&=\sum_{n=0}^{\infty}\left(p(n)+\sum_{r=1}^{\infty}p(n-tr(2rm-m+2))-\sum_{s=1}^{\infty}p(n-t(2s-1)(sm-m+1))\right)q^n.
 	\end{align*}
 	Thus, for all non-negative integers $n$, we have 
 	\begin{align}\label{eqn-NEW-1}
 	p_{mt,t}(n)= p(n)+\sum_{r=1}^{\infty}p(n-tr(2rm-m+2))-\sum_{s=1}^{\infty}p(n-t(2s-1)(sm-m+1)).
 	\end{align}  
 	This completes the proof of the theorem.	
 \end{proof}
In the following theorem, we prove that $p_{mt,t}(n)$ satisfies Ramanujan-type congruences, and these congruences follow from those satisfied by the ordinary partition function $p(n)$.  
 \begin{theorem}\label{thm4}
 	Let $k,m, a\geq 1$ and $b$ be integers. Suppose  that $p(an+b)\equiv 0 \pmod k$ for all non-negative integers $n$. Then, for all $t\geq 1$, we have 
 	\begin{align*}
 	p_{mat,at}(an+b)\equiv 0 \pmod k
 	\end{align*}
 	for all non-negative integers $n$.
 \end{theorem}
 \begin{proof}
 Let $n\geq 0$. From \eqref{thm3.1}, we obtain 
 	\begin{align}\label{eqn-NEW-2}
 	p_{mat,at}(an+b)&=p(an+b)+\sum_{r=1}^{\infty}p(a(n-tr(2rm-m+2))+b)\notag\\
 	&-\sum_{s=1}^{\infty}p(a(n-t(2s-1)(sm-m+1))+b).
 	\end{align}
 	We note that the terms remaining in the sums in \eqref{eqn-NEW-1} satisfy that $n-tr(2rm-m+2)$ and $n-t(2s-1)(sm-m+1)$ are non-negative. Hence, the same is true in \eqref{eqn-NEW-2}. Now, if $p(\ell a+b)\equiv 0\pmod{k}$ for every non-negative integer $\ell$, then \eqref{eqn-NEW-2} yields that $p_{mat, at}(an+b)\equiv 0\pmod{k}$. This completes the proof. 
 \end{proof}
 As an application of Theorem \ref{thm4}, we find that $p_{mt,t}(n)$ satisfies the Ramanujan's famous congruences for certain infinite families of $t$ and $m$. Much to Ramanujan's credit, the ``Ramanujan congruences'' for $p(n)$ are given below. If $k\geq 1$, then for every non-negative integer $n$, we have 
 \begin{align*}
 p\left(5^{k} n+\delta_{5, k}\right) & \equiv 0 \pmod {5^{k}}; \\
 p\left(7^{k} n+\delta_{7, k}\right) & \equiv 0 \pmod {7^{[k / 2]+1}}; \\
 p\left(11^{k} n+\delta_{11, k}\right) & \equiv 0 \pmod{ 11^{k}},
 \end{align*}
 where $\delta_{p, k}:=1/24 \bmod {p^k}$ for $p= 5, 7,11$.  In the following, we prove that $p_{mat,at}(n)$ satisfy the Ramanujan congruences when $a=5^k, 7^k, 11^k$.
 \begin{corollary}\label{ramanujan-cong}
 For all $k, t\geq 1, m\geq1$, and for every non-negative integer $n$, we have  
 \begin{align*}
 p_{m\cdot5^{k}t,5^{k}t}\left(5^{k} n+\delta_{5, k}\right) & \equiv 0 \pmod {5^{k}}; \\
 p_{m\cdot7^{k}t,7^{k}t}\left(7^{k} n+\delta_{7, k}\right) & \equiv 0 \pmod {7^{[k / 2]+1}}; \\
 p_{m\cdot11^{k}t,11^{k}t}\left(11^{k} n+\delta_{11, k}\right) & \equiv 0 \pmod{ 11^{k}}.\\
 \end{align*}
 \end{corollary}
\begin{proof}
Combining Ramanujan congruences for $p(n)$ and Theorem \ref{thm4} we readily obtain that $p_{mat,at}(n)$ satisfies the Ramanujan congruences when $a=5^k, 7^k, 11^k$. This completes the proof.
\end{proof}

%\bibliographystyle{acm}
%\bibliographystyle{sej}
%\bibliographystyle{plain}
%\bibliography{ref.cray.bib}
\end{document}